\documentclass[preprint,12pt]{elsarticle}
\usepackage[cp1251]{inputenc}

\usepackage{a4wide,amsfonts, amsmath, amscd}
\usepackage[psamsfonts]{amssymb}

\usepackage{amssymb}
 \usepackage{graphicx}
\usepackage{epsfig}

\usepackage{amsthm}

\newcommand{\sysn}{\left\{\begin{array}{rcl}}
\newcommand{\sysk}{\end{array}\right.}

\newtheorem{theorem}{Theorem}[section]
\newtheorem{lemma}[theorem]{Lemma}

\theoremstyle{example}
\newtheorem{example}[theorem]{Example}

\newtheorem{proposition}[theorem]{Proposition}
\theoremstyle{definition}
\newtheorem{definition}[theorem]{Definition}
\newtheorem{remark}[theorem]{Remark}
\newtheorem{corollary}[theorem]{Corollary}

\journal{...}

\begin{document}

\title{On countable tightness type properties of spaces of quasicontinuous functions}

\author{Alexander V. Osipov}

\address{Krasovskii Institute of Mathematics and Mechanics, \\ Ural Federal
 University, Yekaterinburg, Russia}

\ead{OAB@list.ru}

\begin{abstract} In this paper we get characterizations countable tightness, countable fan-tightness and countable strong fan-tightness of spaces of quasicontinuous functions  from an open Whyburn regular space $X$ into the discrete two-point space $\{0,1\}$ with the topology of pointwise convergence through properties of $X$ determined by selection principles. These properties (e.g. $S_1(\mathcal{K}, \mathcal{K})$, $\mathcal{K}_{\Omega}$-Lindel\"{o}fness, $S_1(\mathcal{K}_{\Omega}, \mathcal{K}_{\Omega})$) were defined by M.~Scheepers and studied in theory of selection principles in the class of metric spaces.

For any uncountable cardinal number $\kappa$, we get a functional characterization of $\kappa$-Lusin spaces in class of separable metrizable spaces through tightness of compact subsets of a space of quasicontinuous real-valued functions with the topology of pointwise convergence.

\end{abstract}



\maketitle 


\section{Introduction} A study of some convergence properties in function spaces
is an important task of general topology. The general question in the theory of function spaces is to characterize topological properties of a space of functions on a
topological space $X$.

In $C_p$-theory it have been obtained interested
results on cardinal properties of first-countability, Fr\'{e}chet-Urysohn properties, tightness \cite{arch,Koc1,Koc2,GN,GNS,Sak} of a space $C_p(X,\mathbb{R})$ of continuous real-valued functions on a
Tychonoff space $X$ with the topology of pointwise convergence.

Archangel'skii-Pytkeev theorem \cite{arch} is a nice result
about tightness of function spaces: $t(C_p(X,\mathbb{R}))=\sup\{l(X^n): n\in \mathbb{N}\}$. Thus, $C_p(X,\mathbb{R})$ has countable tightness if
and only if $X^n$ is Lindel\"{o}f for each $n\in \mathbb{N}$.

The following result on countable fan tightness of function
spaces $C_p(X,\mathbb{R})$ is shown  by  A.V.~Archangel'skii \cite{arch2}: $C_p(X,\mathbb{R})$ has countable fan tightness if
and only if $X^n$ is a Menger space for each $n\in \mathbb{N}$ (i.e. $X$ has the property $S_{fin}(\Omega, \Omega)$).

 In \cite{Sak}, M. Sakai is shown that $C_p(X,\mathbb{R})$ has countable strong fan-tightness if
and only if $X$ has the property $S_1(\Omega, \Omega)$.

In papers \cite{os10,os12, os13}, tightness, fan tightness and strong fan-tightness of a space of continuous functions with a set-open (e.g. compact-open) topology were investigated.
In \cite{os11}, we study  tightness type properties of spaces of Baire-one functions with the topology of pointwise convergence.

In this paper, we continue to study  countable tightness, countable fan-tightness and countable strong fan-tightness of spaces of quasicontinuous functions with the topology of pointwise convergence.

A function $f:X\rightarrow Y$ is {\it quasicontinuous} at $x$ if for any open set $V$ containing $f(x)$ and any $U$ open containing $x$, there exists a nonempty open set
$W\subseteq U$ such that $f(W)\subseteq V$. It is {\it quasicontinuous} if it is quasicontinuous at
every point.
Call a set semi-open (or quasi-open) if it is contained in the closure of its
interior. Then $f: X\rightarrow Y$ is quasicontinuous if and only if the inverse of every
open set is quasi-open.

Quasicontinuous functions were studied in
many papers, see for examples \cite{Bor,HM,Hol1,Hol2,Hol3,Hol4}.

Levine \cite{Lev} studied quasicontinuous maps under the name of
semi-continuity using the terminology of semi-open sets.  A
function $f:X\rightarrow Y$ is called {\it semi-continuous} if
$f^{-1}(V)$ is semi-open in $X$ for every open
 set $V$ of $Y$. A map $f:X\rightarrow \mathbb{R}$ is quasicontinuous if and only
if $f$ is semi-continuous \cite{Lev}.

\medskip

Let $X$ and $Y$ be Hausdorff  topological spaces, $Q_p(X, Y)=(Q(X, Y),\tau_p)$ be
the space of all  quasicontinuous functions on $X$ with values in
$Y$ and $\tau_p$ be the pointwise convergence topology.

\medskip

\section{Preliminaries}

All spaces under consideration are assumed to be regular. A subset $U$ of a topological space $X$ is called a {\it regular
open set} or an {\it open domain} if $U=Int\overline{U}$ holds. A
subset $F$ of a topological space $X$ is called a {\it regular
closed } set or a {\it closed domain } if $F =\overline{IntF}$
holds.



A set $A$ is called {\it minimally bounded} with respect to the
topology $\tau$ in a topological space $(X,\tau)$ if
$\overline{Int A}\supseteq A$ and $Int\overline{A}\subseteq A$
(\cite{1}, p.101). Clearly this means $A$ is semi-open and
$X\setminus A$ is semi-open. In the case of {\it open} sets,
minimal boundedness coincides with regular openness.

Note that if $U$ is a  minimally bounded (e.g. regular open) set
of $X$ such that $U$ is not dense subset in $X$ and $B\subset
\overline{U}\setminus U$ then there is a quasicontinuous function
$f:X\rightarrow \mathbb{R}$ such that $f(U\cup B)=0$ and
$f(X\setminus (U\cup B))=1$ (see Lemma 4.2 in \cite{Hol5}).

Let us recall some properties of a
topological space $X$.

(1) A space $X$ is {\it Fr\'{e}chet-Urysohn} provided that for
every $A\subset X$ and $x\in \overline{A}$ there exists a sequence
in $A$ converging to $x$.

(2)  A space $X$ has {\it countable tightness} at a point $x$ (denoted $t(x, X)=\omega$) if $x\in
\overline{A}$, then $x\in \overline{B}$ for some countable $B\subseteq A$. A space $X$ has countable tightness (denoted $t(X)=\omega$) if $t(x, X)=\omega$ for every $x\in X$.

(3) A space $X$ has {\it countable fan-tightness} at a point $x$ (denoted $vet(x, X)=\omega$)  if for any countable family $\{A_n : n\in\omega\}$
of subsets of $X$ satisfying $x\in \bigcap_{n\in \omega} \overline{A_n}$ it is possible to select finite sets $K_n\subset A_n$ in such a way that $x\in \overline{\bigcup_{n\in\omega} K_n}$.
A space $X$ has countable fan-tightness (denoted $vet(X)=\omega$) if $vet(x, X)=\omega$ for every $x\in X$.

(4)  A space $X$ is said to have {\it countable
strong fan-tightness} at a point $x$ (denoted $vet_1(x, X)=\omega$) if for each countable family $\{A_n: n\in \omega\}$
of subsets of $X$ such that $x\in \bigcap_{n\in \omega} \overline{A_n}$, there exist $a_i\in A_i$ such that
$x\in \overline{\{a_i: i\in\omega\}}$. A space $X$ has countable strong fan-tightness (denoted $vet_1(X)=\omega$) if $vet_1(x, X)=\omega$ for every $x\in X$.



(5) A space $X$ is said to be  {\it open
Whyburn} if for every open set $A\subset X$ and every $x\in
\overline{A}\setminus A$ there is an open set $B\subseteq A$ such
that $\overline{B}\setminus A=\{x\}$ \cite{Os1}.

\medskip
Note that the class of open
Whyburn spaces is quite wide; for example, it includes all first countable regular spaces \cite{Os1} and, therefore, all metrizable spaces.

\medskip

Let $X$ be a Tychonoff topological space, $C(X,\mathbb{R})$ be the
space of all  continuous functions on $X$ with values in
$\mathbb{R}$ and $\tau_p$ be the pointwise convergence topology.
Denote by $C_p(X,\mathbb{R})$ the topological space
$(C(X,\mathbb{R}), \tau_p)$.

A real-valued function $f$ on a space $X$ is a {\it Baire-one
function} (or a {\it function of the first Baire class}) if $f$ is
a pointwise limit of a sequence of continuous functions on $X$.

 The symbol $\bf{0}$
stands for the constant function to $0$. A basic open neighborhood
of $\bf{0}$ in $\mathbb{R}^X$  is of the form $[F, (-\epsilon, \epsilon)]=\{f\in
\mathbb{R}^X: f(F)\subset (-\epsilon, \epsilon)\}$, where $F\in
[X]^{<\omega}$ and $\epsilon>0$.

\medskip
Let us recall that a cover $\mathcal{U}$ of a set $X$ is called

$\bullet$ an {\it $\omega$-cover} if each finite set $F\subseteq
X$ is contained in some $U\in \mathcal{U}$;

$\bullet$ a {\it $\gamma$-cover} if for any $x\in X$ the set
$\{U\in \mathcal{U}: x\not\in U\}$ is finite.

\medskip

In this paper $\mathcal{A}$ and $\mathcal{B}$ will be collections
of the following  covers of a space $X$:


$\mathcal{O}^s$ : the collection of all semi-open covers of $X$.

$\Omega$ : the collection of open $\omega$-covers of $X$.

$\mathcal{K}$: the collection $\mathcal{U}$ of open subsets of $X$ such that $X=\bigcup\{\overline{U}: U\in \mathcal{U}\}$.


$\Omega^s$ : the collection of minimally bounded  $\omega$-covers
of $X$.

$\Gamma^s$ : the collection of minimally bounded  $\gamma$-covers
of $X$.

$\mathcal{K}_{\Omega}$ is the set of $\mathcal{U}$ in $\mathcal{K}$ such that no element of $\mathcal{U}$ is dense in $X$, and for each finite set $F\subseteq X$, there is a $U\in \mathcal{U}$ such that $F\subseteq \overline{U}$.

$\mathcal{K}_{\Gamma}$ is the set of $\mathcal{U}$ in $\mathcal{K}$ such that no element of $\mathcal{U}$ is dense in $X$, and $\{\overline{U}: U\in \mathcal{U}\}$ is a $\gamma$-cover of $X$.

\begin{definition} Let $\mathcal{P}$ be a collection of covers of $X$.
A space is {\it $\mathcal{P}$-Lindel\"{o}f} if each element of $\mathcal{P}$ has a countable subset in $\mathcal{P}$.
\end{definition}

\begin{definition}(\cite{Kun})
A Hausdorff space $X$ is called a {\it Lusin space} ({\it in the
sense of Kunen})  if

(a) Every nowhere dense set in $X$ is countable;

(b) $X$ has at most countably many isolated points;

(c) $X$ is uncountable.
\end{definition}

If $X$ is an uncountable Hausdorff space then $X$ is $\mathcal{O}^s$-Lindel\"{o}f (semi-Lindel\"{o}f) if
and only if $X$ is a Lusin space (Corollary 2.5 in \cite{GJR}).

\medskip

If $X$ is a Lusin space, $X$ is hereditarily
Lindel\"{o}f (Lemma 1.2 in \cite{Kun}). Hence, if $X$ is a regular Lusin space then $X$ is perfect normal (3.8.A.(b) in \cite{Eng}).

\medskip

If $X$ is a Lusin space, so is every uncountable subspace (Lemma 1.1 in \cite{Kun}).


\medskip

Many topological properties are defined or characterized in terms
 of the following classical selection principles (see \cite{H1}).
 Let $\mathcal{A}$ and $\mathcal{B}$ be sets consisting of
families of subsets of an infinite set $X$. Then:

$S_{1}(\mathcal{A},\mathcal{B})$ is the selection hypothesis: for
each sequence $\{A_{n}: n\in \mathbb{N}\}$ of elements of
$\mathcal{A}$ there is a sequence $\{b_{n}\}_{n\in \mathbb{N}}$
such that for each $n$, $b_{n}\in A_{n}$, and $\{b_{n}:
n\in\mathbb{N} \}\in \mathcal{B}$.

$S_{fin}(\mathcal{A},\mathcal{B})$ is the selection hypothesis:
for each sequence $\{A_{n}: n\in \mathbb{N}\}$ of elements of
$\mathcal{A}$ there is a sequence $\{B_{n}\}_{n\in \mathbb{N}}$ of
finite sets such that for each $n$, $B_{n}\subseteq A_{n}$, and
$\bigcup_{n\in\mathbb{N}}B_{n}\in\mathcal{B}$.

\medskip

In \cite{Sh2},  M.~Scheepers investigated combinatoric properties (e.g. $S_1(\mathcal{K}, \mathcal{K})$, $\mathcal{K}_{\Omega}$-Lindel\"{o}fness, $S_1(\mathcal{K}_{\Omega}, \mathcal{K}_{\Omega})$, $S_{fin}(\mathcal{K}, \mathcal{K})$) in the class of separable metric spaces.
Unexpectedly, it turned out that these properties are characterized by countable tightness type properties of spaces of quasicontinuous functions.
Observe that every $T_2$ countable space $X$ satisfies all these properties and therefore Theorems 3.1, 4.1 and 4.3  are true for arbitrary countable spaces.

\medskip

For other notation and terminology almost without exceptions we follow the Engelking's book \cite{Eng}.

\section{Countable tightness}

\begin{lemma}\label{lem1} Every uncountable open Whyburn $\mathcal{K}$-Lindel\"{o}f space is a Lusin space.
\end{lemma}

\begin{proof}
$(1)$ Claim that every nowhere dense set in $X$ is countable.

 Since the closure of a nowhere dense subset in $X$ is a nowhere dense set, we can
consider only closed nowhere dense sets in $X$.

   Assume that $A$ is an uncountable closed nowhere dense set in
$X$. Since $X$ is open Whyburn, for every point $a\in A$ there is
a regular open set $O_a\subseteq X\setminus A$ such that
$\overline{O_a}\setminus (X\setminus A)=\{a\}$.

Consider the family $\gamma=\{O(x): x\in X\}$ of open sets of $X$ where $O(x)=O_x$ for $x\in A$ and $O(x)$ is an open neighborhood of $x$ such that $\overline{O(x)}\cap A=\emptyset$ for $x\not\in A$. Then $\gamma\in \mathcal{K}$, but $\gamma'\not\in \mathcal{K}$ for any countable subfamily $\gamma'\subset \gamma$.

\medskip

$(2)$ Claim that $X$ has at most countably many isolated points.

Assume that $X$ has uncountable many isolated points $D$.

 Consider the set $W=Int\overline{D}$. Since $X$ is open Whyburn, for every point $d\in W\setminus D$
there is an open subset $O_d\subseteq D$ such that
$\overline{O_d}\setminus D=\{d\}$.

(a) Suppose that for every point $d\in W\setminus D$ there is a
neighborhood $V_d$ of $d$ such that  $|O_d\cap V_d|\leq \omega$.
Let $W_d=O_d\cap V_d$. Then $\overline{W_d}\setminus D=\{d\}$,
$W_d\subset D$ and $|W_d|\leq \omega$.

Consider the open family $\mu=\{\{d\}: d\in D\}\cup \{W_d:
d\in W\setminus D\}\cup (X\setminus \overline{D})$. Note that $\mu\in \mathcal{K}$, but $\mu'\not\in \mathcal{K}$ for any countable subfamily $\gamma'\subset \gamma$.

$(b)$ Suppose that there is a point $d\in W\setminus D$ such that
$|O_d\cap V_d|>\omega$ for every neighborhood $V_d$ of $d$. Let
$O_d=O_1\cup O_2$ such that $O_1\cap O_2=\emptyset$ and
$|O_i|>\omega$ for $i=1,2$.

Let $d\in \overline{O_2}$. Then, we consider the open family $\sigma=\{\{x\}: x\in O_1\}\cup \{O_2, X\setminus \overline{O_d}\}$. Note that $\sigma\in \mathcal{K}$, but $\sigma'\not\in \mathcal{K}$ for any countable subfamily $\sigma'\subset \sigma$.

\end{proof}

\begin{remark} In \cite{Sh2},  M.~Scheepers proved that a separable metrizable space $X$ is Lusin if and only if it is $\mathcal{K}$-Lindel\"{o}f.
Note that every Lusin space is hereditarily
Lindel\"{o}f (Lemma 1.2 in \cite{Kun}), hence, every Lusin space is $\mathcal{K}$-Lindel\"{o}f.
\end{remark}

\begin{corollary} An uncountable open Whyburn space $X$ is $\mathcal{K}$-Lindel\"{o}f if and only if it is a Lusin space.

\end{corollary}

It is well known that $f$ is
of the first Baire class if and only if $f^{-1}(U)$ is a countable unions of zero sets
for every open $U\subseteq \mathbb{R}$ (see Exercise
3.A.1 in \protect\cite{lmz1}).

\begin{proposition} Let $X$ be a Lusin space. Then every real-valued quasicontinuous function is of the first Baire class.
\end{proposition}

\begin{proof} Let $X$ be a Lusin space. Then $X$ is a perfect normal space and, hence, any open set is a countable unions of zero sets. It remains to note that any semi-open set in $X$ is a unions of open set and (countable set of points) countable nowhere dense subset of $X$.
\end{proof}

\begin{corollary} Let $X$ be a Lusin space and $f$ be a real-valued quasicontinuous function. Then the set $D_f=\{x\in X: f$ is discontinuous in $x \}$ is countable.
\end{corollary}

\begin{proof}  Fix an open countable basis $\{V_n\}$ for $\mathbb{R}$. We then have

   $x\in D_f$ $\Leftrightarrow$ $\exists$ $n$ $[x\in f^{-1}(V_n)\setminus Int(f^{-1}(V_n))]$, i.e.,

   $D_f=\bigcup\{  f^{-1}(V_n)\setminus Int(f^{-1}(V_n)): n\in \mathbb{N}\}$. It remains to note that any set $f^{-1}(V_n)\setminus Int(f^{-1}(V_n))$ is countable.
\end{proof}

The space of all quasicontinuous functions from $X$ into the discrete space $\mathbb{D}=\{0,1\}$ is denote by $Q_p(X, \mathbb{D})$.

\begin{theorem}\label{th2} For an uncountable open Whyburn space $(X,\tau)$ the following statements are equivalent:

\begin{enumerate}

\item $X$ is $\Omega^s$-Lindel\"{o}f;

\item $X$ is $\mathcal{K}_{\Omega}$-Lindel\"{o}f;

\item $t({\bf 0}, Q_p(X,\mathbb{R}))=\omega$;

\item $t(f, Q_p(X,\mathbb{R}))=\omega$ for every $f\in C(X,\mathbb{R})$;

\item $t(Q_p(X, \mathbb{D}))=\omega$.

\end{enumerate}

\end{theorem}

\begin{proof} $(1)\Rightarrow(2)$ and $(4)\Rightarrow(3)$. It is trivial.

$(2)\Rightarrow(1)$. It is enough to prove that for any $\mathcal{U}=\{U_{\alpha}\}_{\alpha\in A}\in \Omega^s$ there exists $\mathcal{V}=\{V_{\beta}\}_{\beta\in B}\in \mathcal{K}_{\Omega}$ such that for any $\beta\in B$ there is $\alpha\in A$ such that $\overline{V_{\beta}}\subseteq U_{\alpha}$. We can denote this as $\mathcal{V}\succ \mathcal{U}$.  Let $F$ be a finite subset of $X$. Then there is $U_{\alpha}$ such that $F\subseteq U_{\alpha}$. Since $X$ is an open Whyburn space, there is an open set $V_F$ such that $F\subseteq \overline{V_F}\subseteq U_{\alpha}$. Let $\mathcal{V}=\{V_F: F\in [X]^{<\omega}\}$.

$(2)\Rightarrow(3)$. Assume that ${\bf 0}\in \overline{\{f_{\alpha}: \alpha\in A\}}$ where $|A|>\omega$. Let $n\in \mathbb{N}$ and $\mathcal{V}_n=\{V_{\alpha,n}=f^{-1}_{\alpha}((-\frac{1}{n},\frac{1}{n})): \alpha\in A\}$.
Then, $\mathcal{V}_n$ is a semi-open $\omega$-cover of $X$. Since $X$ is open
Whyburn, for every $V\in \mathcal{V}_n$ and a finite subset $F$ of $V$ there is an open subset $W_{F,V}$ in $X$ such that $F\subseteq \overline{W_{F,V}}\subseteq V$. Then $\mathcal{W}=\{W_{F,V}: V\in \mathcal{V}, F\in [V]^{<\omega}\}\in \mathcal{K}_{\Omega}$ and $\mathcal{W}\succ \mathcal{V}_n$.
Since $X$ is $\mathcal{K}_{\Omega}$-Lindel\"{o}f, there is a countable subfamily $\mathcal{W}'=\{W_{F_i,V_{\alpha_i,n}}: i\in \mathbb{N}\}$ of $\mathcal{W}$  such that $\mathcal{W}'\in \mathcal{K}_{\Omega}$.
It follows that $\mathcal{V}'_n=\{V_{\alpha_i,n}: i\in \mathbb{N}\}$ is a countable subfamily of $\mathcal{V}_n$. Denote by $F_n=\{f_{\alpha_i}: i\in \mathbb{N}\}$.
Thus, for every $n\in \mathbb{N}$, $\mathcal{V}'_n\in \Omega^s$ which implies ${\bf 0}\in \overline{\bigcup\{F_n:  n\in \mathbb{N}\}}$.

$(3)\Rightarrow(1)$.  Let $\{U_{\alpha}\}_{\alpha\in A}\in \Omega^s$.
Consider the quasicontinuous function $f_{\alpha}:X\rightarrow \{0,1\}$ such that $f_{\alpha}(U_{\alpha})=0$ and $f_{\alpha}(X\setminus U_{\alpha})=1$ for each $\alpha\in A$. Then ${\bf 0}\in \overline{\{f_{\alpha}: \alpha\in A\}}$. Since $t({\bf 0},Q(X,\mathbb{R}))=\omega$, there is $B\subset A$ such that $|B|=\omega$ and ${\bf 0}\in \overline{\{f_{\alpha}: \alpha\in B\}}$. It follows that $\{U_{\alpha}: \alpha\in B\}\in \Omega^s$.

$(3)\Rightarrow(4)$. Note that for any space $X$ and maps $f,g: X\rightarrow \mathbb{R}$ such
that $f$ is continuous and $g$ is quasicontinuous, the map $f+g: X\rightarrow \mathbb{R}$ defined by
$(f+g)(x)=f(x)+g(x)$ is quasicontinuous (Proposition 5.4 in
\cite{KT}). Thus, the mapping $h_f: Q_p(X,\mathbb{R})\rightarrow Q_p(X,\mathbb{R})$ such that $h_f(g)=f+g$ for every $g\in Q_p(X,\mathbb{R})$ is a homeomorphism for any $f\in C_p(X,\mathbb{R})$.
It follows that $(3)$ implies $(4)$.

$(2)\Rightarrow(5)$.  Let $f\in Q_p(X, \mathbb{D})$. Note that $f^{-1}(\{d\})$ is a semi-open set in $X$ for every $d\in \mathbb{D}$.
The set $D_f=\{x\in X: f$ is discontinuous in $x \}$ is a nowhere dense subset of $X$. Since $X$ is Lusin, the set $\overline{D_f}$ is countable.

Consider the new topology $\tau_f$, the base of which forms the family $\tau\cup \{\{d\}: d\in \overline{D_f}\}$.
Let $id: (X,\tau_f)\rightarrow (X,\tau)$ be the identity mapping from $(X,\tau_f)$ onto $(X,\tau)$.

It's easy to check that if $g\in Q_p((X,\tau),\mathbb{D})$ then $g\circ id\in Q_p((X,\tau_f),\mathbb{D})$.

\medskip
Claim that $(X,\tau_f)$ is $\mathcal{K}_{\Omega}$-Lindel\"{o}f. To do this, we will prove two facts for $\mathcal{K}_{\Omega}$-Lindel\"{o}f spaces.

\medskip

(a) {\it If $X$ is $\mathcal{K}_{\Omega}$-Lindel\"{o}f and $G$ is an open subset of $X$ then $G$ is $\mathcal{K}_{\Omega}$-Lindel\"{o}f.}

\medskip
By Lemma \ref{lem1}, $X$ is a Lusin space and, hence, $X$ is a perfect normal space. Thus the set $X\setminus G$ is $G_{\delta}$.
Let $X\setminus G=\bigcap W_i$ where $W_{i+1}\subset W_{i}$ and $W_i\in \tau$ for each $i\in \mathbb{N}$. Consider $\mathcal{V}=\{V_{\alpha}: \alpha\in A\}\in \mathcal{K}_{\Omega}$ where $\mathcal{K}_{\Omega}$ in the subspace $G$. Note that $V_{\alpha}$ is not dense in $G$ for each $V_{\alpha}\in \mathcal{V}$.  Since $X$ is regular, there is an open set $O_{\alpha}$ in $X$ such that $\overline{O_{\alpha}}\subset G\setminus \overline{V_{\alpha}}$.

Let $\mathcal{O}_i=\{V_{\alpha,i}=V_{\alpha}\cup (W_i\setminus \overline{O_{\alpha}}) : \alpha\in A\}$. Note that $V_{\alpha,i}$ is not dense in $(X,\tau)$ for each $\alpha\in A$.
Then $\mathcal{O}_i\in \mathcal{K}_{\Omega}$ in the space $(X,\tau)$.
Then, there exists $\mathcal{O}'_i=\{V_{\alpha_j}\cup (W_i\setminus \overline{O_{\alpha_j}}): j\in \mathbb{N}\}$ such that $\mathcal{O}'_i\in \mathcal{K}_{\Omega}$ in the space $(X,\tau)$.

Let $\mathcal{V}'=\{V_{\alpha_j(i)}: i,j\in \mathbb{N}\}$.
Remain note that $\mathcal{V}'\in \mathcal{K}_{\Omega}$ where $\mathcal{K}_{\Omega}$ in the subspace $G$.
If $F\in [G]^{<\omega}$ then there is $i'\in \mathbb{N}$ such that $F\cap W_{i'}=\emptyset$. Hence, there is $j'$ such that $F\subseteq \overline{V_{\alpha_{j'}(i')}}$.

\medskip

(b) {\it If $X$ is an open $\mathcal{K}_{\Omega}$-Lindel\"{o}f subspace of $X\cup S$ where $S$ is countable then $X\cup S$ is $\mathcal{K}_{\Omega}$-Lindel\"{o}f.}

\medskip

We can assume that $X\cap S=\emptyset$ otherwise we can consider $S'=S\setminus X$. Let $S=\{s_n : n\in \mathbb{N}\}$.
Consider $\mathcal{V}=\{V_{\alpha}: \alpha\in A\}\in \mathcal{K}_{\Omega}$ where $\mathcal{K}_{\Omega}$ in the space $X\cup S$.

Let $\mathcal{V}_n=\{V_{\alpha}\in \mathcal{V}: \{s_1,...,s_n\}\subseteq \overline{V_{\alpha}}$ and $V_{\alpha}\cap X\neq \emptyset\}$.

\medskip

(1) Assume that for any $n\in \mathbb{N}$ there is $k(n)>n$ and $V_{\alpha(k(n))}\in \mathcal{V}_{k(n)}$ such that $X\subset \overline{V_{\alpha(k(n))}}$. Then $\{V_{\alpha(k(n))}: n\in \mathbb{N}\}\in \mathcal{K}_{\Omega}$.

\medskip

(2) Otherwise there is $n'\in \mathbb{N}$ such that for any $k>n'$ and $V_{\alpha}\in \mathcal{V}_k$ the set $X\setminus \overline{V_{\alpha}}$ is not empty.

Thus $\mathcal{U}_n=\{X\cap V_{\alpha}: V_{\alpha}\in \mathcal{V}_n\}\in \mathcal{K}_{\Omega}$ in the space $X$ for every $n>k$.
Then, there is $\mathcal{U}'_n=\{X\cap V_{\alpha_i}: i\in \mathbb{N}\}\in \mathcal{K}_{\Omega}$ in the space $X$ for every $n>k$. Note that $P_n=\{V_{\alpha_i}\in \mathcal{U}'_n: i\in \mathbb{N}\}\in \mathcal{K}_{\Omega}$ in the space $X\cup \{s_1,...,s_n\}$.
Let $P=\bigcup P_n$. Then $P$ is countable, $P\subset \mathcal{V}$ and $P\in\mathcal{K}_{\Omega}$ in the space $X\cup S$.

\medskip

By the fact (a),  the subspace $X\setminus \overline{D_f}$ is $\mathcal{K}_{\Omega}$-Lindel\"{o}f.

By the fact (b), the space $(X,\tau_f)$ is $\mathcal{K}_{\Omega}$-Lindel\"{o}f.

\medskip Assume that $f\in \overline{\{f_{\alpha}: \alpha\in A\}}$ where $F=\{f_{\alpha}: \alpha\in A\}\subset Q_p((X,\tau), \mathbb{D})$ and $|A|>\omega$.
Then $f \circ id\in \overline{\{f_{\alpha} \circ id: \alpha\in A\}}$ where $\{f_{\alpha} \circ id: \alpha\in A\}\subset Q_p((X,\tau_f), \mathbb{D})$.
Note that (2) implies (4), $(X,\tau_f)$ is $\mathcal{K}_{\Omega}$-Lindel\"{o}f and $f \circ id\in C((X,\tau_f), \mathbb{R})$.
Then, by (4), there is a countable set $B\subset A$ such that  $f \circ id\in \overline{\{f_{\alpha_i} \circ id: \alpha_i\in B\}}$. It follows that $f\in \overline{\{f_{\alpha_i}: \alpha_i\in B\}}$.

$(5)\Rightarrow(1)$.  Similar to the implication $(3)\Rightarrow(1)$.

\end{proof}

In particular, we get the following corollary in class of metrizable spaces.

\begin{corollary} A metrizable space $X$ is $\mathcal{K}_{\Omega}$-Lindel\"{o}f if, and only if, $t(Q_p(X, \mathbb{D}))=\omega$.
\end{corollary}



\section{Countable strong fan-tightness and countable fan-tightness}

\begin{theorem}\label{th1} For an uncountable open Whyburn space $X$ the following statements are equivalent:

\begin{enumerate}

\item $X$ satisfy $S_1(\Omega^s,\Omega^s)$;

\item $X$ satisfy $S_1(\mathcal{K}_{\Omega},\mathcal{K}_{\Omega})$;

\item $vet_1({\bf 0}, Q_p(X, \mathbb{R}))=\omega$;

\item $vet_1(f, Q_p(X,\mathbb{R}))=\omega$ for every $f\in C(X,\mathbb{R})$;

\item $vet_1(Q_p(X, \mathbb{D}))=\omega$.

\end{enumerate}

\end{theorem}

\begin{proof} $(1)\Rightarrow(2)$ and $(4)\Rightarrow(3)$. It is trivial.

$(2)\Rightarrow(1)$. Let $\mathcal{U}_i=\{U^i_{\alpha}\}_{\alpha\in A_i} \in \Omega^s$ for each $i\in \mathbb{N}$. In Theorem \ref{th2} ($(2)\Rightarrow(1)$),  we proved for any $\mathcal{U}=\{U_{\alpha}\}_{\alpha\in A}\in \Omega^s$ there exists $\mathcal{V}=\{V_{\beta}\}_{\beta\in B}\in \mathcal{K}_{\Omega}$ such that for any $\beta\in B$ there is $\alpha\in A$ such that $\overline{V_{\beta}}\subseteq U_{\alpha}$, i.e.  $\mathcal{V}\succ \mathcal{U}$.
Thus, for every $i\in \mathbb{N}$ there is $\mathcal{V}_i\in \mathcal{K}_{\Omega}$ such that $\mathcal{V}_i\succ \mathcal{U}_i$.
By (2), there is $V^i_{\beta_i}\in \mathcal{V}_i$ for each $i\in \mathbb{N}$ such that $\{V^i_{\beta_i}: i\in \mathbb{N}\}\in \mathcal{K}_{\Omega}$.
For every $\beta_i$ there is $\alpha_i$ such that $\overline{V^i_{\beta_i}}\subset U^i_{\alpha_i}$.  It follows that $\{U^i_{\alpha_i}: i\in \mathbb{N}\}\in \Omega^s$.

$(3)\Rightarrow(1)$.  Let $\mathcal{U}_n=\{U^n_{\alpha}\}_{\alpha\in A_n}\in \Omega^s$ for each $n\in \mathbb{N}$.
Consider the quasicontinuous function $f_{\alpha,n}:X\rightarrow \{0,1\}$ such that $f_{\alpha,n}(U^n_{\alpha})=0$ and $f_{\alpha,n}(X\setminus U^n_{\alpha})=1$ for each $\alpha\in A_n$ and $n\in \mathbb{N}$. Then ${\bf 0}\in \overline{\{f_{\alpha,n}: \alpha\in A_n\}}$ for each $n\in \mathbb{N}$. Since $vet_1({\bf 0}, Q_p(X, \mathbb{R}))=\omega$, there is $f_{\alpha_n,n}\in \{f_{\alpha,n}: \alpha\in A_n\}$ for each $n\in \mathbb{N}$ such that
${\bf 0}\in \overline{\{f_{\alpha_n,n}: n\in \mathbb{N}\}}$. It follows that $\{U^n_{\alpha_n}: n\in \mathbb{N}\}\in \Omega^s$.

$(1)\Rightarrow(3)$. Let $X$ has the property $S_1(\Omega^s,\Omega^s)$. Then $X$ is  $\Omega^s$-Lindel\"{o}f and, by Theorem \ref{th2}, $X$ is $\mathcal{K}_{\Omega}$-Lindel\"{o}f.
Consider a countable family $\{A_n: n\in \mathbb{N}\}$ of subsets of $Q_p(X,\mathbb{R})$ such that ${\bf 0}\in \bigcap_{n\in \mathbb{N}} \overline{A_n}$. For every $n\in \mathbb{N}$ we consider $\mathcal{V}_n=\{V_{n,i,f}=f^{-1}((-\frac{1}{i},\frac{1}{i})):  i\in \mathbb{N}$ and $i\geq n, f\in A_n\}$. Since ${\bf 0}\in \overline{A_n}$ , the family $\mathcal{V}_n$ is a semi-open $\omega$-cover of $X$.

Since $X$ is an open Whyburn regular space, there is $\mathcal{U}_n\in \mathcal{K}_{\Omega}$ such that $\mathcal{U}_n\succ \mathcal{V}_n$ for each $n\in \mathbb{N}$.
By implication ($(1)\Rightarrow(2)$), for each $n\in \mathbb{N}$ there is $U_{n,\beta_n}\in \mathcal{U}_n$ such that $\{U_{n,\beta_n}: n\in \mathbb{N}\}\in \mathcal{K}_{\Omega}$.
For each $n\in \mathbb{N}$ there are $i_n$ and $f_n$ such that $\overline{U_{n,\beta_n}}\subseteq V_{n,i_n,f_n}$. Hence, $\{V_{n,i_n,f_n}: n\in \mathbb{N}\}$ is an $\omega$-cover of $X$.  Then, we consider the set $\{f_n : n\in \mathbb{N}\}$.

(1) $f_n\in A_n$ for each $n\in \mathbb{N}$.

(2) ${\bf 0}\in \overline{\{f_n : n\in \mathbb{N}\}}$.

Let $K\in [X]^{<\omega}$ and $\epsilon>0$ and $[K,\epsilon]=\{f\in Q_p(X,\mathbb{R}) : f(K)\subset (-\epsilon, \epsilon)\}$.

Then, there is $n'$ such that $\frac{1}{i_{n'}}<\epsilon$ and $K\subseteq V_{n',i_{n'},f_{n'}}$. It implies that $f_{n'}\in [K,\epsilon]$.

$(3)\Rightarrow(4)$. Similarly ($(3)\Rightarrow(4)$) in Theorem \ref{th2}.

$(2)\Rightarrow(5)$.  Let $f\in Q_p(X, \mathbb{D})$. Note that $f^{-1}(\{d\})$ is a semi-open set in $X$ for every $d\in \mathbb{D}$. Thus, $D_f$ is countable nowhere dense subset of $X$. Since $X$ is Lusin, the set $\overline{D_f}$ is countable.

Similarly the proof of $((2)\Rightarrow(5))$ in Theorem \ref{th2}, we consider the new topology $\tau_f$, the base of which forms the family $\tau\cup \{\{d\}: d\in \overline{D_f}\}$.

It's easy to check (almost the same as in Theorem \ref{th2}) that  the space $(X,\tau_f)$ has the property  $S_1(\mathcal{K}_{\Omega},\mathcal{K}_{\Omega})$ and $f\in C((X,\tau_f), \mathbb{R})$.
Then, by (4), $vet_1(f, Q_p(X, \{0,1\})=\omega$.

$(5)\Rightarrow(1)$.  Similar to the implication $(3)\Rightarrow(1)$.

\end{proof}

\begin{corollary} A metrizable space $X$ is $S_1(\mathcal{K}_{\Omega},\mathcal{K}_{\Omega})$ if, and only if,  $vet_1(Q_p(X,\mathbb{D}))=\omega$.
\end{corollary}

Similar to the proof of Theorem \ref{th1}, we can prove the following theorem.

\begin{theorem} For an uncountable open Whyburn space $X$ the following statements are equivalent:

\begin{enumerate}

\item $X$ satisfy $S_{fin}(\Omega^s,\Omega^s)$;

\item $X$ satisfy $S_{fin}(\mathcal{K}_{\Omega},\mathcal{K}_{\Omega})$;

\item $vet({\bf 0}, Q_p(X, \mathbb{R}))=\omega$;

\item $vet(f, Q_p(X,\mathbb{R}))=\omega$ for every $f\in C(X,\mathbb{R})$;

\item $vet(Q_p(X, \mathbb{D}))=\omega$.

\end{enumerate}

\end{theorem}

\begin{corollary} A metrizable space $X$ is $S_{fin}(\mathcal{K}_{\Omega},\mathcal{K}_{\Omega})$ if, and only if,  $vet(Q_p(X,\mathbb{D}))=\omega$.
\end{corollary}

\section{Tightness of compact subsets}

Let $\kappa$ be an unfinite cardinal number. Let $\{X_{\lambda}: \lambda\in A\}$ be a family of topological spaces.
Let $X=\prod_{\lambda\in A} X_{\lambda}$  be the Cartesian product with the Tychonoff topology. Take a point $p=(p_{\lambda})_{\lambda\in A}\in X$. For each
$x=(x_{\lambda})_{\lambda\in A}\in X$, let $Supp(x)=\{\lambda\in A: x_{\lambda}\neq p_{\lambda}\}$. Then the
subspace $\Sigma_{\kappa}(p)=\{x\in X: |Supp(x)|\leq\kappa\}$
of $X$ is called a {\it $\Sigma_{\kappa}$ -product} of $\{X_{\lambda}: \lambda\in A\}$ about $p$ ($p$ is
called the base point).

In (\cite{kommal}, Proposition 1), A.P. Kombarov and V.I. Malykhin  proved that

$(\bullet)$ if $t(\prod\limits_{i=1}^n X_{\alpha_i})\leq \kappa$ for every $n\in \mathbb{N}$ and a finite family $\alpha_1,...,\alpha_n\in A$ then $t(\Sigma_{\kappa}(p))\leq \kappa$.

\medskip
Suppose that $\kappa$ is a cardinal number.
A separable metrizable space $X$ is a {\it $\kappa$-Lusin set} if $|X|\geq \kappa$
and, for every meager set $M$, we have $|X\cap M|<\kappa$. Usually, $\aleph_1$-Lusin sets and $2^{\omega}$-Lusin sets are called Lusin sets and $\mathfrak{c}$-Lusin sets, respectively. Every Lusin set is also $\mathfrak{c}$-Lusin. Moreover, if Continuum
Hypothesis $(CH)$ holds, then every $\mathfrak{c}$-Lusin set is also a Lusin set. However, it
is consistent that these notions are not equivalent. Indeed, e.g., under Martin's
Axiom $(MA)$ and the failure of $CH$ there are $\mathfrak{c}$-Lusin sets on $\mathbb{R}$ which are not Lusin \cite{5}.

\medskip
If the axiom of choice holds, then every cardinal $\kappa$ has a successor, denoted $\kappa^+$, where $\kappa^+>\kappa$ and there are no cardinals between $\kappa$ and its successor.

\begin{theorem}\label{th3} Let $\kappa$ be an uncountable cardinal number.  A separable metrizable space $X$ of cardinality $\geq \kappa$  is a $\kappa$-Lusin set if and only if $t(K)<\kappa$  for every compact subset $K$ of $Q_p(X,\mathbb{R})$.
\end{theorem}

\begin{proof} $(\Rightarrow)$. Let $A$ be a countable dense subset of a $\kappa$-Lusin space $X$. Note that if $g,f\in Q_p(X,\mathbb{R})$ and  $g(x)=f(x)$ for every $x\in A$ then $\{x\in X: g(x)\neq f(x)\}\subseteq D_g\cup D_f$ where $D_h$ is a set of discontinuous points of a function $h$. Since $X$ is $\kappa$-Lusin, $|D_g\cup D_f|<\kappa$ and we get that  $|\{x\in X: g(x)\neq f(x)\}|<\kappa$.

 Let $K$ be a compact subset of $Q_p(X,\mathbb{R})$. Consider the projection function
 $p=\pi_A: Q_p(X,\mathbb{R})\rightarrow \mathbb{R}^A$, i.e., $p(f)=f\mid A$ for every $f\in Q_p(X,\mathbb{R})$. Since $\mathbb{R}^A$ is metrizable, the set $p(K)$ is a metrizable compact space.
Let $z\in p(K)$. Then $S_z=p^{-1}(z):=\{f\in Q_p(X,\mathbb{R}): f|A=z\}$ is closed in $Q_p(X,\mathbb{R})$. Let $\tilde{z}\in S_z$. Then
$S_z\subset \Sigma_{\kappa}(\tilde{z})$  where  $\Sigma_{\kappa}(\tilde{z}):=\{h\in \mathbb{R}^X: |\{x\in X: h(x)\neq \tilde{z}(x)\}|<\kappa \}$.
By ($\bullet$), $t(\Sigma_{\kappa}(\tilde{z}))<\kappa$. It follows that $t(S_z\cap K)<\kappa$ for every $z\in p(K)$ and $K=\bigcup\{S_z\cap K : z\in p(K)\}$. By Theorem 6 in \cite{arch1} (If $f:X\rightarrow Y$ is a continuous closed mapping then $t(X)\leq\sup \{t(Y), t(f^{-1}(y)): y\in Y\}$), we get that $t(K)<\kappa$.

$(\Leftarrow)$. Assume that $t(K)<\kappa$  for every compact subset $K$ of $Q_p(X,\mathbb{R})$ and $X$ is not $\kappa$-Lusin. Then there exists a closed nowhere dense subset $A$ of $X$ such that $|A|\geq \kappa$.

 Let $B\subset A$. Then, there is $f_B: X\rightarrow \mathbb{D}$ be a quasicontinuous function such that $f_B(B)=1$ and $f_B(A\setminus B)=0$.

  Indeed, let $O$ be an open set in $X$ such that $\overline{O}\setminus O\supseteq A$ and $X\setminus \overline{O}\neq \emptyset$.

  Then $f_B(x)=1$ for $x\in B\cup (\overline{O}\setminus A)$ and $f_B(x)=0$ for other $x\in X$.

  Note that $f_{B'}|(X\setminus A)=f_{B''}|(X\setminus A)$ for any $B', B''\subset A$.

 It is clear that $K=\{f_B: B\subset A\}$ is homeomorphic to the compact space $2^{A}$. But, $t(2^{A})=t(K)\geq\kappa$, it is a contradiction.

\end{proof}

\begin{corollary} A uncountable separable metrizable space $X$ is Lusin if, and only if, $t(K)=\omega$  for every compact subset $K$ of $Q_p(X,\mathbb{R})$.
\end{corollary}

\begin{corollary} If $Q_p(X,\mathbb{R})$ is homeomorphic to $Q_p(Y,\mathbb{R})$ where $X$ is $\kappa$-Lusin, then $Y$ is $\kappa$-Lusin, too.
\end{corollary}





\section{Examples}

In \cite{Os1}, it is proved that if $X$ is a metric space then  $Q_p(X,\mathbb{R})$ is Fr\'{e}chet-Urysohn at the point ${\bf 0}$ if, and only if, $X$ is countable.
The following example shows that for a countable tightness and even for a countable strong fan-tightness of the space $Q_p(X,\mathbb{R})$, a space $X$ can be uncountable.

Given some special axioms, one can show that there are uncountable separable metrizable space $X$ such that $t(Q_p(X, \mathbb{D}))=\omega$. In particular: The axiom $(\diamond)$ asserts that there is a sequence $(S_{\alpha}: \alpha<\omega_1)$  such that

(1) For each $\alpha$, $S_{\alpha}\subset \alpha$, and

(2) For every subset $A$ of $\omega_1$, the set $\{\alpha<\omega_1: A\cap \alpha=S_{\alpha}\}$  is stationary.

\medskip

It is well known that the axiom $(\diamond)$ is consistent relative to the consistency of classical mathematics and implies but is not equivalent to the Continuum Hypothesis.

\begin{example}$(\diamond)$ There exists a Lusin space $X$ such that $vet_1({\bf 0}, Q_p(X,\mathbb{R}))=\omega$.
\end{example}

In (\cite{Sh2}, Theorem 5), M.~Scheepers constructed an example of a uncountable separable metrizable space $X$ which has the property $S_{1}(\mathcal{K}_{\Omega},\mathcal{K}_{\Omega})$. By Theorem \ref{th1}, we get an example with the required properties.

\begin{example}$(\diamond)$ There exists a Lusin space $X$ such that $t(Q_p(X,\mathbb{R}))>\omega$.
\end{example}

In (\cite{2}, see ref.[2] in \cite{Sh2}), W. Just proved that if there is any Lusin set at all, then there is a Lusin set which is not $\mathcal{K}_{\Omega}$-Lindel\"{o}f. By Theorem \ref{th2}, we get an example with the required properties.

\begin{example} $(MA+\neg CH)$ For each cardinal $\kappa\leq 2^{\omega}$ with $cf(\kappa)>\omega$ there is a separable metric space $X$ such that $t(K)\leq\kappa$ for each compact subset $K$ of $Q_p(X,\mathbb{R})$ and $t(C)>\omega$ for some compact subset $C$ of $Q_p(X,\mathbb{R})$.
\end{example}

Under Martin's
Axiom $(MA)$ and the failure of $CH$ for each cardinal $\kappa\leq 2^{\omega}$ with $cf(\kappa)>\omega$ there are $\kappa$-Lusin sets in $\mathbb{R}$ which are not Lusin \cite{5}.
By Theorem \ref{th3}, we get an example with the required properties.

\medskip

\section{Remark}

The idea of defining a new topology $\tau_f$ for a quasicontinuous function $f$, which we use in Theorems \ref{th2} and \ref{th1}, can be easily used for the Fr\'{e}chet-Urysohn property of space $Q_p(X, \mathbb{D})$.
Thus, combining the results of the article \cite{Os1}, we obtain the following theorem.

\begin{theorem}\label{th5} For an uncountable open Whyburn space $X$ the following statements are equivalent:

\begin{enumerate}

\item $X$ satisfy $S_{1}(\Omega^s,\Gamma^s)$;

\item $X$ satisfy $S_{1}(\mathcal{K}_{\Omega},\mathcal{K}_{\Gamma})$;

\item  $Q_p(X,\mathbb{R})$ is Fr\'{e}chet-Urysohn at the point ${\bf 0}$;

\item $Q_p(X,\mathbb{R})$ is Fr\'{e}chet-Urysohn at the point $f$ for every $f\in C(X,\mathbb{R})$;

\item $Q_p(X, \mathbb{D})$ is Fr\'{e}chet-Urysohn.

\end{enumerate}

\end{theorem}

By Theorem 3.11 in \cite{Os1}, Theorem \ref{th5}, Theorem 4.1 in \cite{Hol5} and Theorem 4.6 in \cite{KT} we get the following result.

\begin{corollary} Let $X$ and $Y$ be nontrivial metrizable spaces. Then the following are equivalent:

\begin{enumerate}

\item  $X$ is countable;

\item $Q_p(X, \mathbb{D})$ is Fr\'{e}chet-Urysohn;

\item $Q_p(X, Y)$ is Fr\'{e}chet-Urysohn;

\item $Q_p(X, Y)$ is first countable;

\item $Q_p(X, Y)$ is metrizable.

\end{enumerate}

\end{corollary}

In \cite{Sh2}, it is proved that a metrizable space $X$ is Lusin if, and only if, it is $\mathcal{K}$-Lindel\"{o}f. Obviously, $\mathcal{K}_{\Omega}$-Lindel\"{o}fness implies $\mathcal{K}$-Lindel\"{o}fness of space. Let us note however that Kunen (Theorem 0.0. in \cite{Kun}) has shown that under $(MA+\neg CH)$ there are no Lusin spaces at all.

Thus, in the class of metrizable spaces we get the following result.

\begin{corollary}  $(MA+\neg CH)$ Let $X$ and $Y$ be nontrivial metrizable spaces. Then the following are equivalent:

\begin{enumerate}

\item  $X$ is countable;

\item $Q_p(X, \mathbb{D})$ is metrizable;

\item $Q_p(X, Y)$  is Fr\'{e}chet-Urysohn;

\item $Q_p(X, Y)$ is first countable;

\item $Q_p(X, Y)$ is metrizable;

\item $Q_p(X, Y)$ has countable tightness;

\item $Q_p(X, Y)$ has countable  fan-tightness;

\item $Q_p(X, Y)$ has countable strong fan-tightness.

\end{enumerate}

\end{corollary}

\section{Open questions}

{\bf Question 1.} Could it be that some $Q_p(X, \mathbb{D})$ has countable tightness (countable  fan-tightness, countable strong fan-tightness, is Fr\'{e}chet-Urysohn) but none $Q_p(X,\mathbb{R})$  has this property?

\medskip
In (\cite{Sh2}, Problem 3),  M. Scheepers asks: {\it Could it be that some Lusin set is  $\mathcal{K}_{\Omega}$-Lindel\"{o}f, but none has property $S_{1}(\mathcal{K}_{\Omega}, \mathcal{K}_{\Omega})$}?

\medskip

This question can be divided into two sub-questions in a functional context.

\medskip

{\bf Question 2.} Is there a $T_2$-space $X$ such that $Q_p(X, \mathbb{D})$ has countable tightness  but none $Q_p(X, \mathbb{D})$ has countable fan-tightness?

\medskip

{\bf Question 3.} Is there a $T_2$-space $X$ such that $Q_p(X, \mathbb{D})$ has countable fan-tightness  but none $Q_p(X, \mathbb{D})$  has countable strong fan-tightness?

\medskip

{\bf Acknowledgements.} I would like to thank Evgenii Reznichenko and the
referee for careful reading and valuable comments.

\bibliographystyle{model1a-num-names}
\bibliography{<your-bib-database>}

\end{document}